\documentclass[letterpaper, 10 pt, conference]{ieeeconf}  
\IEEEoverridecommandlockouts
\usepackage{graphicx}
\usepackage{amsmath,amssymb}
\usepackage{subfig}

\DeclareMathOperator*{\argmin}{argmin}
\newcommand{\set}[1]{\left\{#1\right\}}
\newcommand{\goesto}{\rightarrow}

\newcommand{\Bcal}{\mathcal{B}}
\newcommand{\Dcal}{\mathcal{D}}
\newcommand{\Jcal}{\mathcal{J}}
\newcommand{\Scal}{\mathcal{S}}
\newcommand{\Vcal}{\mathcal{V}}
\newcommand{\Xcal}{\mathcal{X}}

\newcommand{\rbar}{\overline{r}}

\newcommand{\Rbb}{\mathbb{R}}
\newcommand{\Zbb}{\mathbb{Z}}

\newcommand{\E}{\mathbb{E}}

\newcommand{\Ind}{\mathbb{I}}

\newtheorem{definition}{Definition}
\newtheorem{proposition}{Proposition}

\newtheorem{theorem}{Theorem}
\newtheorem{lemma}{Lemma}

\usepackage{color}

\title{\LARGE \bf Service Rate Control For Jobs with Decaying Value}

\author{Neal Master and Nicholas Bambos% <-this % stops a space
  \thanks{Neal Master is supported by the Department of Defense (DoD)
    through the National Defense Science \& Engineering Graduate
    Fellowship (NDSEG) Program.}% <-this % stops a space
  \thanks{N. Master and N. Bambos are with the Department of
    Electrical Engineering, Stanford University, Stanford, CA, 94305,
    USA.  {\tt\small\{nmaster, bambos\}@stanford.edu}} }

\begin{document}
\maketitle
\thispagestyle{empty}
\pagestyle{empty}
\begin{abstract}
  The task of completing jobs with decaying value arises in a number
  of application areas including healthcare operations, communications
  engineering, and perishable inventory control. We consider a system
  in which a single server completes a finite sequence of jobs in
  discrete time while a controller dynamically adjusts the service
  rate. During service, the value of the job decays so that a greater
  reward is received for having shorter service times. We incorporate
  a non-decreasing cost for holding jobs and a non-decreasing cost on
  the service rate. The controller aims to minimize the total cost of
  servicing the set of jobs. We show that the optimal policy is
  non-decreasing in the number of jobs remaining -- when there are
  more jobs in the system the controller should use a higher service
  rate. The optimal policy does not necessarily vary monotonically
  with the residual job value, but we give algebraic conditions which
  can be used to determine when it does. These conditions are then
  simplified in the case that the reward for completion is constant
  when the job has positive value and zero otherwise. These algebraic
  conditions are interesting because they can be verified without
  using algorithms like value iteration and policy iteration to
  explicitly compute the optimal policy. We also discuss some future
  modeling extensions.
\end{abstract}

\section{Introduction\label{sec:intro}}
There are a variety of queueing applications for which job completion
rewards decay over time. For example, this is the case in healthcare
systems. In some situations, the patients can be treated like ``jobs''
and the decaying ``reward'' is the decaying patient health --
patients' health will typically decay as treatment is delayed and this
can reduce the efficacy of medical procedures
\cite{Mcquillan_ICU_1998}. Jobs can also represent diagnostic tests. A
study showed that a majority of primary care physicians were
dissatisfied with delays in viewing test results and that these delays
can lead to further delays in treatment
\cite{Poon_MedicalTests_2004}. The negative impact of patient
mortality motivates the general study of queueing for jobs with
decaying value.

There are also applications in communications engineering. A notable
example is that of multimedia streaming over wireless. Each packet is
a job which is completed when the packet is successfully transmitted
over a noisy channel. For the sake of maintaining a high quality user
experience, multimedia traffic requires low latency as well as low
jitter.  The real-time nature of streaming means that the packets
rapidly decay to having zero value. This has led to a number of
interesting practical and theoretical problems in the wireless
communications literature. One key problem is that of packet
scheduling for downlink cellular systems. In these systems, cellular
base-stations need to schedule many different traffic streams while
taking into account channel conditions in order to maintain high
quality-of-service (QoS) for all users \cite{Dua_CD2_2010}. In other
contexts, delay sensitive service becomes relevant for transmitter
power control with constraints on inter-departure times
\cite{Master_HOL_2014}. Higher transmitter power gives a higher
probability of successful packet transmission so there is a natural
trade-off between power usage and delay.

A third application area is that of perishable inventory control. Food
items can be modeled as ``jobs'' while the process of selling to
consumers can be modeled as ``service''. For example, food items will
decay with time as they eventually spoil, at which point they have no
value. In these models, the value of food items will decay differently
under varying storage and service conditions giving rise to many
scheduling and service rate control problems. See
\cite{Nahmias_Perishable_1982} for a survey.

Aside from applications oriented research, there is a considerable
body of theoretical work geared towards queueing systems for jobs with
decaying value. In \cite{Dalal_Impatient_2005}, ``impatient'' users in
an M/M/1 queue are scheduled under the constraint that the rewards for
servicing each user decay exponentially. Stochastic depletion problems
cover a broad range of preemptive scheduling problems in which items
are processed while the rewards for doing so decay over time. In
\cite{Chan_Depletion_2009}, greedy scheduling policies for such
problems are shown to be suboptimal by no more than a factor of 2.

In this paper, we consider the following type of system: A finite set
of identical jobs are sequentially serviced by a single server in
discrete time. The controller chooses the probability that the current
head-of-line (HOL) job will reach completion in the current time
slot. When a job reaches the server, it has an initial value. This
value decays during service and the controller gains a positive reward
(i.e. negative cost) when the service is completed. When the value of
the job reaches zero, the job is ejected from the system. Non-negative
costs are incurred in each time slot for holding the residual jobs as
well as for the choice of service probability. We seek to minimize the
total cost incurred for servicing the set of jobs.

One of the unique features of this model is that the value decay only
occurs during service. This is motivated by several specific
applications. In wireless streaming, we have previously considered a
similar model in which the value decay follows a step function so that
jobs essentially have service time constraint
\cite{Master_HOL_2014}\cite{Master_HOL_2015}. The idea is that when
multimedia is streamed over wireless, it is important to maintain a
regular stream of information. Because information is encoded across
packets, it can be better to drop packets and degrade the quality of
the stream rather than delay the entire stream; the service time
constraints enforce this behavior. In perishable inventory control,
having decay during service but not during storage models the idea
that decay happens on different time scales. For example, the quality
of certain food items decay very slowly (practically not at all) if
stored properly but will decay rapidly during transportation and
processing.

Because we focus on this specific type of value decay, this work
expands on and partially complements the existing literature. For
instance, others have studied monotonicity properties of the optimal
service rate control policy for a continuous time Markovian queue with
jobs whose value does not decay
\cite{George_ServiceRateControl_2001}. In the operations research
community, there has also been work on myopic policies for
non-preemptive scheduling of jobs whose value decays over the entire
sojourn time rather than just during service
\cite{Chan_Decay_2015}. Note that a model in which job value decays
during the entire sojourn time does not encompass the problem of
having job value decay only during service.

The remainder of the paper is organized as follows. In
Sec. \ref{sec:model}, we mathematically define the aforementioned
system. This allows us to formulate the problem in a dynamic
programming \cite{Bertsekas_DP_2005} framework. In
Sec. \ref{sec:numerical}, we numerically demonstrate some of the
salient structural features of optimal policies. In particular, we
comment on monotonicity of the policies as the number of jobs
decreases and as the HOL job value decreases. In
Sec. \ref{sec:theory}, we prove sufficient (and in some cases also
necessary) conditions for these observed monotonicity properties to
hold. We identify future areas of research in Sec. \ref{sec:future}
and conclude in Sec. \ref{sec:conclusion}.

\section{System Model and Optimal Control\label{sec:model}}
In this section, we mathematically define the system of interest. We
describe the dynamics as well as the costs. We formulate the optimal
control in a dynamic programming \cite{Bertsekas_DP_2005} framework
and use some results on stochastic shortest path problems
\cite{Bertsekas_SSP_1991} to show that optimal policies exist.

A finite set of $B \in \Zbb_{> 0}$ identical jobs is sequentially
served in discrete time indexed by $t \in \Zbb_{\geq 0}$. The number
of jobs in the system in time slot $t$ is $b_t$. When a job initially
reaches the head-of-line (HOL) in time slot $t$, it has a value of
$v_t = V \in \Zbb_{> 0}$. In time slot $t$, the HOL job completes
service with probability $s_t \in \Scal \subseteq [0, 1]$ which is
chosen by the controller. If the service is not completed, the value
is decremented by one. The service attempt in time slot $t$ is
independent of all other service attempts. When the HOL job value
reaches zero, the job is ejected from the queue and the next job takes
the HOL. The system terminates when all jobs have either been serviced
or ejected.

Let $\Bcal = \set{1, 2, \hdots, B}$, $\Vcal = \set{1, 2, \hdots, V}$,
and $\Xcal = (\Bcal \times \Vcal) \cup \set{(0, V)}$. The state will
be taken as the remaining number of jobs in the system and the
remaining value of the HOL job so the state at time $t$ is then given
by $(b_t, v_t) \in \Xcal$.  Let $\set{w_t}_{t=0}^\infty$ be an IID
$\text{Uniform}[0,1]$ noise source. We can write the state update
function as follows:
\begin{align*}
  (b_{t + 1}, v_{t + 1})
  &= F(b_t, v_t, s_t, w_t)\\
  &= \left\{
    \begin{array}{ll}
      (b_t, v_t - 1) &; b_t > 0, w_t > s_t, v_t > 1\\
      (b_t - 1, V) &; b_t > 0, w_t > s_t, v_t = 1\\
      (b_t - 1, V) &; b_t > 0, w_t \leq s_t\\
      (0, V) &; b_t = 0
    \end{array}
  \right.
\end{align*}
We assume that $\Scal$ is finite. The set of admissible control
policies is given by
\begin{equation*}
  \Pi = \set{\pi: \Xcal \goesto \Scal}.
\end{equation*}

The cost per time slot of service is $c: \Scal \goesto \Rbb_{\geq
  0}$. The cost per time slot of holding jobs is $h : \Bcal \goesto
\Rbb_{\geq 0}$. The reward for servicing a job is given by $r:\Vcal
\goesto \Rbb_{> 0}$. Therefore, if the HOL job completes service when
it has residual value $v$, the cost is given by $-r(v)$. Although
$r(\cdot)$ is positive and only defined on $\Vcal$, the dynamics
logically suggest that $r(0) = 0$ since jobs with zero value are
ejected. We assume that $c(\cdot)$, $h(\cdot)$, and $r(\cdot)$ are
each non-decreasing. If we let $\Ind_{\set{\cdot}}$ be the indicator
function, we can define the stage cost in time slot $t$ as
\begin{align*}
    G(b_t, v_t, s_t, w_t) 
    &= \Ind_{\set{b_t > 0}}\left(h(b_t) + c(s_t) - \Ind_{\set{w_t \leq s_t}}r(v_t)\right).
\end{align*}

Given the initial state is $(b, v) \in \Xcal$, we define the optimal
cost-to-go as follows:
\begin{align*}
  \Jcal&(b, v) \\
  &= \min_{\pi \in \Pi}
  \E\left[\left.
      \sum_{t=0}^\infty G(b_t, v_t, \pi(b_t, v_t), w_t) \right| (b_0, v_0) = (b, v)\right]
\end{align*}
The system reaches the terminal state $(0, V)$ with probability one in
at most $BV$ time slots.  In addition, $\Bcal$ and $\Vcal$ are finite
so the costs are bounded (though not necessarily
non-negative). Therefore, $\Jcal(b, v)$ is well defined for all $(b,
v) \in \Xcal$.

Because the control policies select probability distributions on the
state transitions, we have a stochastic shortest path problem. By
assumption, $\Scal$ is finite so this can be solved using standard
techniques like value iteration and policy iteration
\cite{Bertsekas_SSP_1991}. Hence, we have the following Bellman
equation
\begin{align*}
  \Jcal&(b, v)
  = \min_{s \in \Scal}\Big\{c(s) + h(b)\\
  &+ s[-r(v) + \Jcal(b-1, V)]\\
  &+ (1-s)[\Jcal(b, v-1)\Ind_{\set{v > 1}} + \Jcal(b-1, V)\Ind_{\set{v=1}}]
  \Big\} 
\end{align*}
with the boundary condition that $\Jcal(0, V) = 0$. In general, there
can be multiple optimal policies but we will refer to \emph{the}
optimal policy as
\begin{align*}
  \mu&(b, v)
  = \min\argmin_{s \in \Scal}\Big\{c(s) + h(b)\\
  &+ s[-r(v) + \Jcal(b-1, V)]\\
  &+ (1-s)[\Jcal(b, v-1)\Ind_{\set{v > 1}} + \Jcal(b-1, V)\Ind_{\set{v=1}}]
  \Big\}
\end{align*}
with $\mu(0, V)$ being arbitrary because $(0, V)$ is a cost-free
trapping state.  Again, since we are solving a stochastic shortest
path problem, $\mu$ can be computed by using either value iteration or
policy iteration \cite{Bertsekas_SSP_1991}.

\begin{figure*}
  \begin{center}
  \subfloat[\label{fig:monotone_give_up}]{
    \includegraphics[width=\columnwidth]{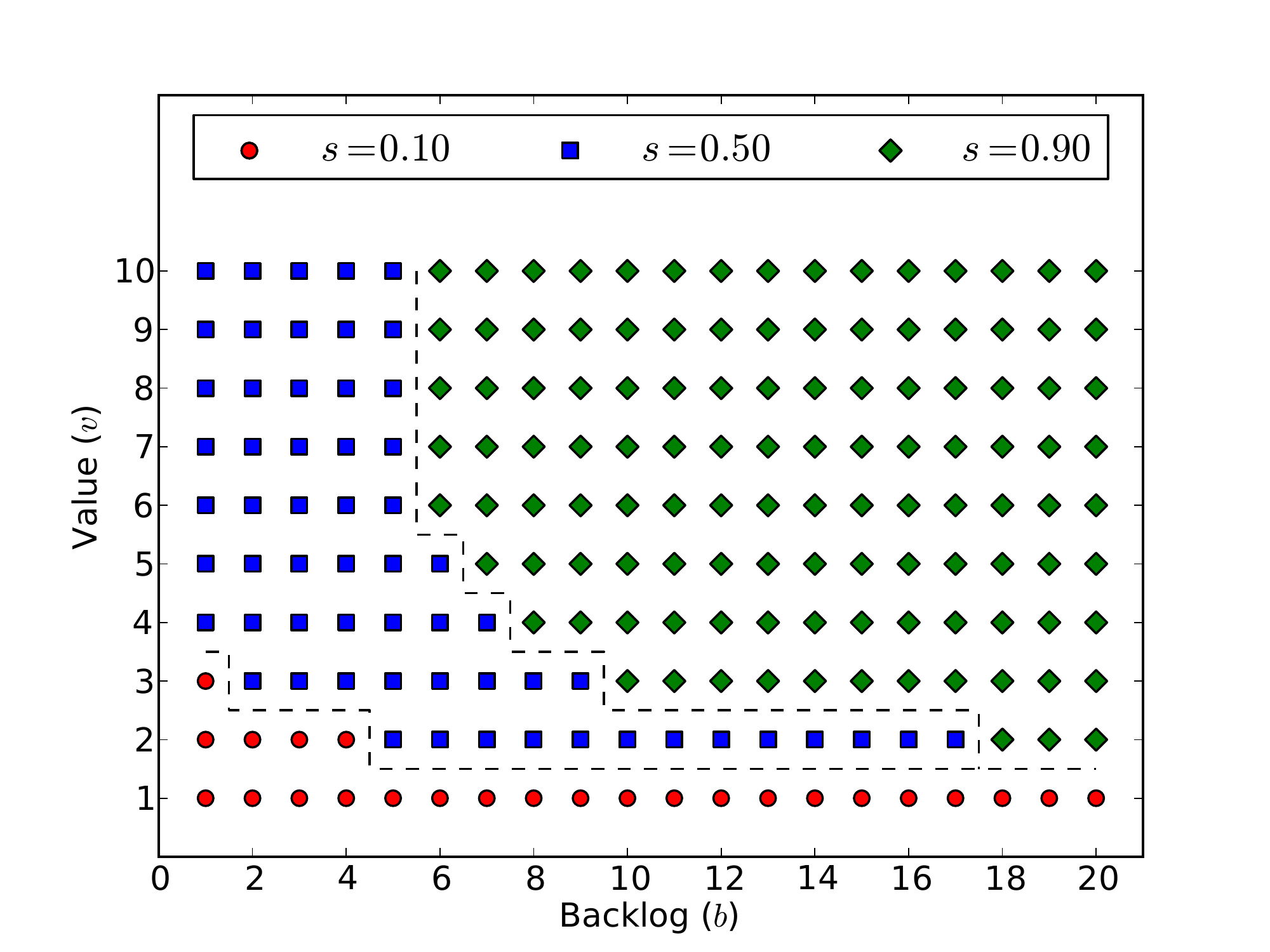}}
  \subfloat[\label{fig:monotone_try_harder}]{
    \includegraphics[width=\columnwidth]{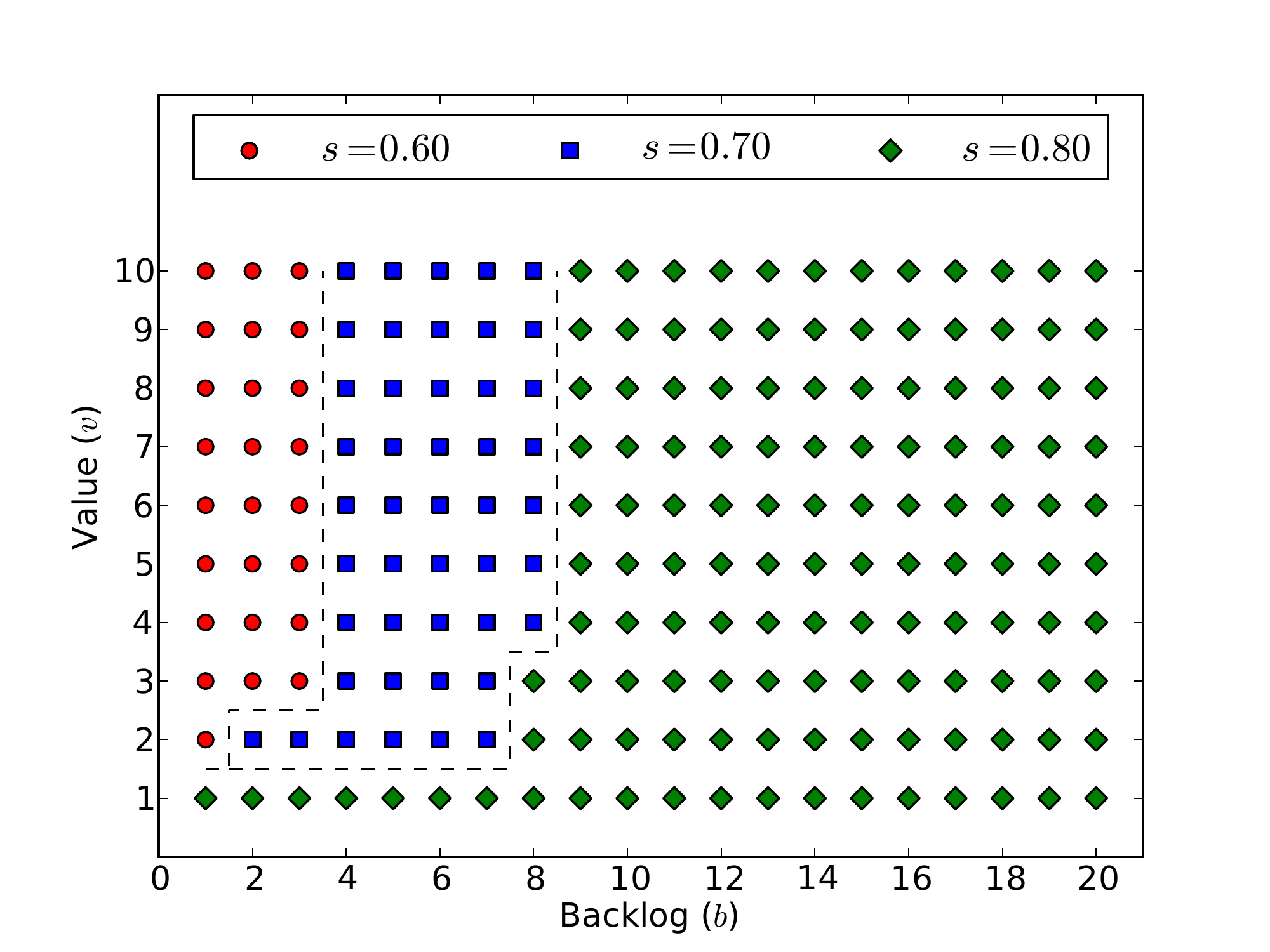}}\\
  \subfloat[\label{fig:monotone_mixed}]{
    \includegraphics[width=\columnwidth]{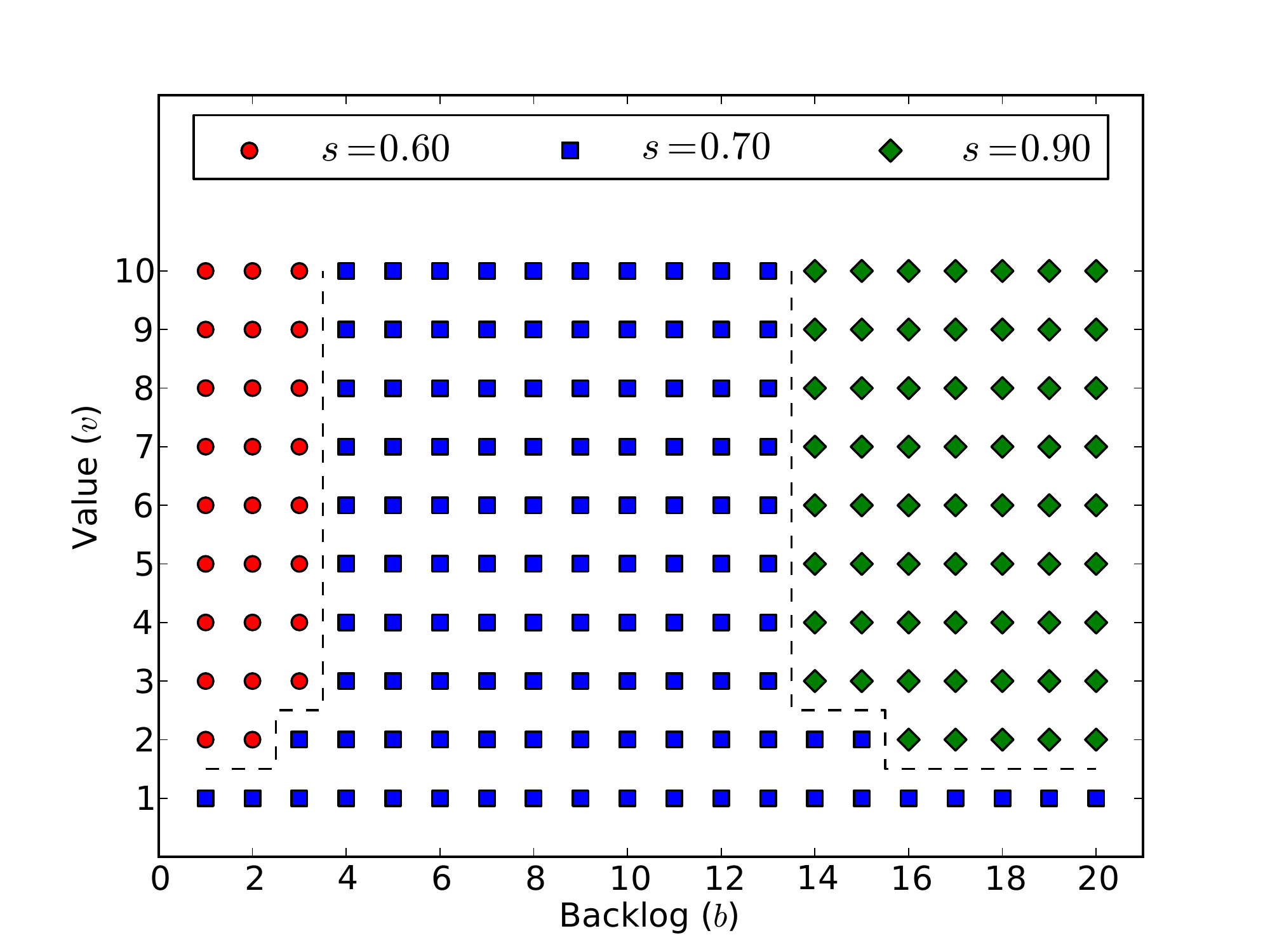}}
  \subfloat[\label{fig:non-monotone}]{
    \includegraphics[width=\columnwidth]{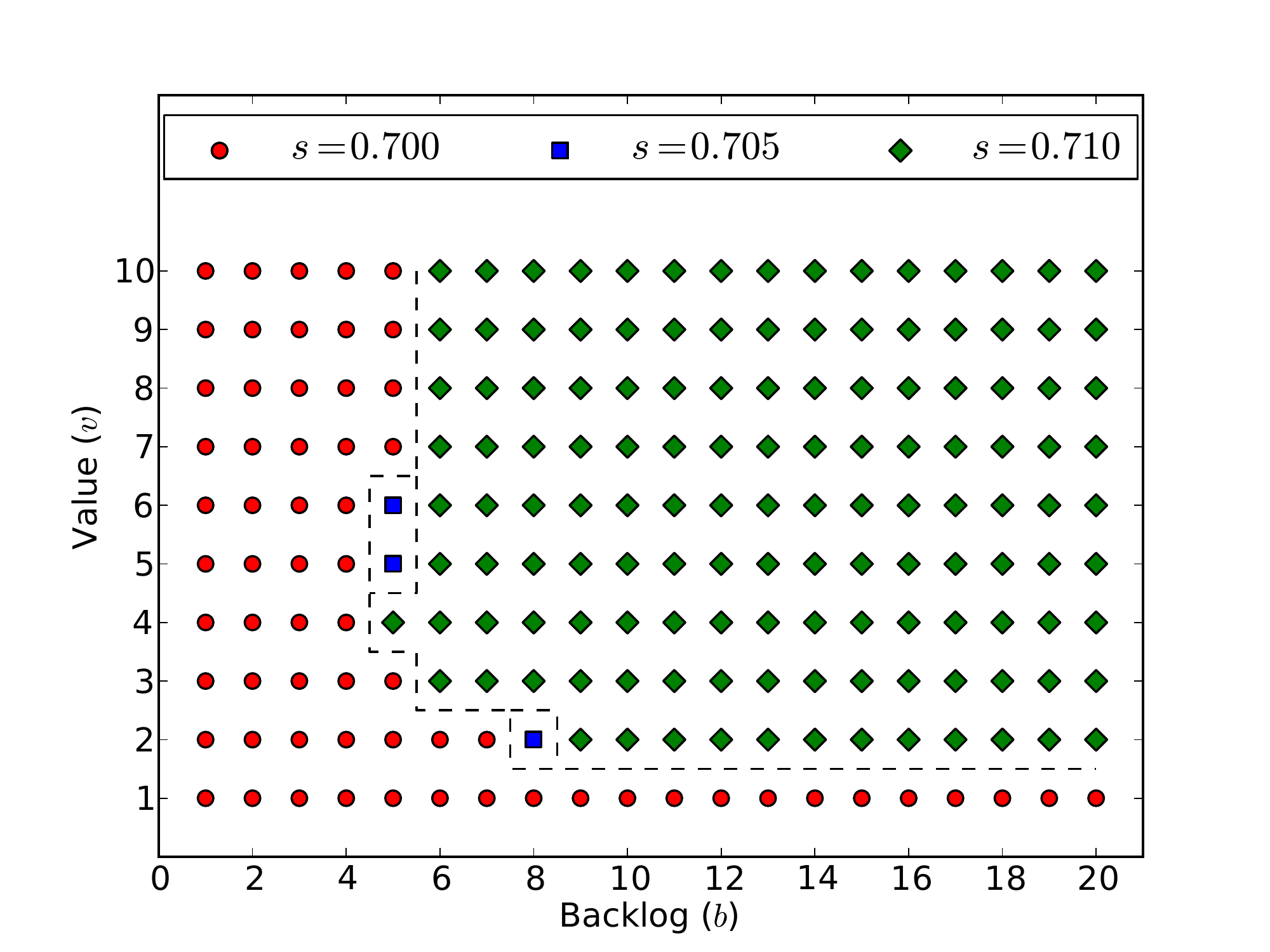}}
  \caption{\footnotesize Examples of $\mu$ for different system
    parameters. For each $(b, v) \in \Bcal \times \Vcal$, we plot a point to 
    indicate the value of $\mu(b, d)$. The dashed lines segment the state space to 
    show when the policy changes. For each of the following policies, we take $h(b) =
    b$, $c(s) = 5\ln\left(\frac{1}{1 - s}\right)$, $V = 10$, and $B =
    20$. 
    For Fig. \ref{fig:monotone_give_up}, $r(v) = v$ and $\Scal = \set{0.1, 0.5, 0.9}$. 
    In this case, $b \mapsto \mu(b, v)$ is non-decreasing for all $v \in \Vcal$ and 
    $v \mapsto \mu(b, v)$ is non-decreasing for all $b \in \Bcal$.
    For Fig. \ref{fig:monotone_try_harder}, 
    $r(v) = \frac{v}{10} + 25$ and $\Scal = \set{0.6, 0.7, 0.8}$. 
    In this case, $b \mapsto \mu(b, v)$ is non-decreasing for all $v \in \Vcal$ and 
    $v \mapsto \mu(b, v)$ is non-increasing for all $b \in \Bcal$.
    For Fig. \ref{fig:monotone_mixed}, 
    $r(v) = \frac{v}{10} + 20$ and $\Scal = \set{0.6, 0.7, 0.9}$. 
    In this case, $b \mapsto \mu(b, v)$ is non-decreasing for all $v \in \Vcal$ while 
    the monotonicity of $v \mapsto \mu(b, v)$ varies with $b$.
    For Fig. \ref{fig:non-monotone}, $r(v) = 5\ln(1 + v)$ 
    and $\Scal = \set{0.700, 0.705, 0.710}$. 
    In this case, $b \mapsto \mu(b, v)$ is non-decreasing for all $v \in \Vcal$ while
    $v \mapsto \mu(b, v)$ is not necessarily monotone in anyway; 
    note that $v \mapsto \mu(5, v)$ is neither non-decreasing nor non-increasing.
  }
  \end{center}
\end{figure*}

\section{Numerical Experiments\label{sec:numerical}}
In this section we offer a brief numerical investigation of the
optimal policy under different conditions. This allows us to
demonstrate the potential structural properties of $\mu$. In each case
we observe that $b \mapsto \mu(b, v)$ is non-decreasing. We observe
that similar monotonicity properties do not always hold for $v \mapsto
\mu(b, v)$. This motivates the analytic investigation in
Sec. \ref{sec:theory}.

For each of the following policies, we take $h(b) = b$, $c(s) =
5\ln\left(\frac{1}{1 - s}\right)$, $V = 10$, and $B = 20$. We vary
$r(\cdot)$ and $\Scal$ to demonstrate different structural
features. Note that even though $c(1) = \infty$, in each example $1
\not\in \Scal$ so the boundedness of $c(\cdot)$ is not violated. These
parameters are not intended to model a specific system and have been
chosen for illustrative purposes.

For Fig. \ref{fig:monotone_give_up}, $r(v) = v$ and $\Scal = \set{0.1,
  0.5, 0.9}$.  In this case, $b \mapsto \mu(b, v)$ is non-decreasing
for all $v \in \Vcal$ and $v \mapsto \mu(b, v)$ is non-decreasing for
all $b \in \Bcal$. To anthropomorphize these properties, we can think
of the server ``giving up'' on a particular job as the job value
decreases. Similarly, the server generally ``tries harder'' when there
are more jobs remaining to be served.

For Fig. \ref{fig:monotone_try_harder}, $r(v) = \frac{v}{10} + 25$ and
$\Scal = \set{0.6, 0.7, 0.8}$.  In this case, $b \mapsto \mu(b, v)$ is
non-decreasing for all $v \in \Vcal$ and $v \mapsto \mu(b, v)$ is
non-increasing for all $b \in \Bcal$. The server still ``tries
harder'' when there are more jobs remaining, but the server also
``tries harder'' as the value of the HOL job decays. This shows that
in some cases, it is optimal for the server to try to complete jobs
even when they have low residual value.

For Fig. \ref{fig:monotone_mixed}, $r(v) = \frac{v}{10} + 20$ and
$\Scal = \set{0.6, 0.7, 0.9}$.  In this case, $b \mapsto \mu(b, v)$ is
non-decreasing for all $v \in \Vcal$ while the monotonicity of $v
\mapsto \mu(b, v)$ varies with $b$. As in the previous two cases, the
server ``tries harder'' when there are more jobs remaining. However,
the monotonicity of $v \mapsto \mu(b, v)$ depends on $b$. This
demonstrates that although it can be optimal for the server to
complete jobs with low residual value, this behavior depends on how
many other jobs are waiting to be served.

For Fig. \ref{fig:non-monotone}, $r(v) = 5\ln(1 + v)$ and $\Scal =
\set{0.700, 0.705, 0.710}$.  In this case, $b \mapsto \mu(b, v)$ is
non-decreasing for all $v \in \Vcal$ while $v \mapsto \mu(b, v)$ is
not necessarily monotone in anyway; note that $v \mapsto \mu(5, v)$ is
neither non-decreasing nor non-increasing. In this final case, we
again see that the server ``tries harder'' when there are more jobs
remaining. However, $v \mapsto \mu(b, v)$ does not exhibit either the
``try harder'' or the ``give up'' behaviors.

\section{Monotonicity of the Optimal Policy\label{sec:theory}}
The numerical examples from the previous section demonstrate the
potentially rich structure of $\mu$. The monotonicity properties that
often hold are interesting because they offer structural insights and
intuitive explanations. However, it is not immediately clear what
conditions are necessary in order to guarantee that these properties
hold. In this section we show that because $h(\cdot)$ is
non-decreasing, $b \mapsto \mu(b, v)$ will be non-decreasing for each
$v \in \Vcal$. We also provide algebraic conditions for determining
the monotonicity of $v \mapsto \mu(b, v)$. These algebraic conditions
are valuable because they can be verified without explicitly solving
for $\mu$. In the case that $r(\cdot)$ is constant, we provide a
simpler algebraic condition which is similar to the one provided in
\cite{Master_HOL_2014}.

We start with some useful definitions.
\begin{definition}
  For each $b \in \Bcal$, let $\delta(b, 0) = 0$ and $\sigma(b, 0) =
  0$. For each $(b, v) \in \Bcal \times \Vcal$, define $\delta(b, v)$
  and $\sigma(b, v)$ as follows:
  \begin{align*}
    \delta(b, v) &= h(b) + 
    \min_{s \in \Scal}\left\{c(s) - s\left[r(v) + \sum_{i=0}^{v-1}\delta(b, i)\right]\right\}\\
    \sigma(b, v) &= \sum_{i=0}^v \delta(b, i)
  \end{align*}
  For each $(b, v) \in \Bcal \times \Vcal$ define $T_{b,v} : \Rbb
  \goesto \Rbb$ as follows:
  \begin{align*}
    T_{b, v}(x) = x + h(b) + \min_{s \in \Scal} \{c(s) - s[r(v) +x]\}
  \end{align*}
\end{definition}

\begin{proposition}
  \label{prop:deltas}
  For each $(b, v) \in \Bcal \times \Vcal$, the Bellman equation can
  be characterized as follows:
  \begin{align*}
    \Jcal(b, v)
    = \left\{
      \begin{array}{ll}
        \Jcal(b-1, V) + \delta(b, 1) &, v = 1\\
        \Jcal(b, v - 1) + \delta(b, v) &, v > 1
      \end{array}
   \right.
  \end{align*}
  Furthermore, the optimal policy can be written as
  \begin{align*}
    \mu(b, v) = \min\argmin_{s \in \Scal}
    \left\{
      c(s) - s\left[ r(v) + \sigma(b, v-1)\right]
    \right\}.
  \end{align*}
\end{proposition}
\begin{proof}
  For any fixed $b \in \Bcal$, we apply the principle of strong
  mathematical induction on $v \in \Vcal$. For $v = 1$, we merely need
  to re-order the Bellman equation:
  \begin{align*}
    \Jcal&(b, 1) \nonumber\\
    &= \min_{s \in \Scal}\Big\{ c(s) + h(b) \\
    &+ s[-r(1) + \Jcal(b-1, V)] + (1 - s)\Jcal(b-1, V)\Big\}\nonumber\\
    &= \Jcal(b-1, V) + h(b) \\
    &+ \min_{s \in \Scal}\Big\{ c(s) + s[-r(1) + \Jcal(b-1, V)] - s\Jcal(b-1, V)\Big\}\nonumber\\
    &= \Jcal(b-1, V) + h(b) + \min_{s \in \Scal}\left\{ c(s) - sr(1)\right\}\\
    &= \Jcal(b-1, V) + \delta(b,1)
  \end{align*}
  We now use this for $v = 2$:
  \begin{align*}
    \Jcal(b, 2)\nonumber
    &= \min_{s \in \Scal}\Big\{ c(s) + h(b) \\
    &+ s[-r(2) + \Jcal(b-1, V)] + (1 - s)\Jcal(b, 1)\Big\}\nonumber\\
    &= \Jcal(b, 1) + h(b) \\
    &+ \min_{s \in \Scal}\left\{ c(s) + s[-r(2) + \Jcal(b-1, V) - \Jcal(b, 1)]\right\}\nonumber\\
    &= \Jcal(b, 1) + h(b) 
    + \min_{s \in \Scal}\left\{ c(s) - s[r(2) + \delta(b, 1)]\right\}\\
    &= \Jcal(b, 1) + \delta(b, 2)
  \end{align*}
  Now assume that the proposition holds for $\set{1, \hdots, v}
  \subsetneq \Vcal$.
  \begin{align*}
    &\Jcal(b, v+1)\\
    &= \min_{s \in \Scal}\Big\{ c(s) + h(b) \\
    &+ s[-r(v+1) + \Jcal(b-1, V)] + (1 - s)\Jcal(b, v)\Big\}\\
    &= \Jcal(b, v) + h(b)\\
    &+\min_{s \in \Scal}\left\{ c(s) + s[-r(v+1) + \Jcal(b-1, V) - \Jcal(b, v)]\right\}\\
    &= \Jcal(b, v) + h(b) + \min_{s \in \Scal}\Big\{ c(s) -s[r(v+1)\\
    &+ \sum_{i=2}^v(\Jcal(b, i) - \Jcal(b, i-1)) + (\Jcal(b, 1)-\Jcal(b-1, V))]\Big\}
   \end{align*}
   Now we apply the induction hypothesis to write sum in the final
   line in terms of $\delta(b, i)$. We then use the definitions of
   $\delta(b, v+1)$ and $\sigma(b, v)$ to complete the proof.
   \begin{align*}
     &\Jcal(b, v+1)\\
    &= \Jcal(b, v) + h(b) 
    + \min_{s \in \Scal}\left\{ c(s)  -s[r(v+1) + \sum_{i=0}^v \delta(b, i)]\right\}\nonumber\\
    &= \Jcal(b, v) + h(b) 
    + \min_{s \in \Scal}\left\{ c(s)  -s[r(v+1) + \sigma(b, v)]\right\}\nonumber\\
    &= \Jcal(b, v) + \delta(b, v + 1)
  \end{align*}

  Now that we have this alternative characterization of the Bellman
  equation, we simply ignore the terms which do not involve $s$ to
  conclude that
  \begin{align*}
    \mu(b, v) = \min\argmin_{s \in \Scal}
    \left\{
      c(s) - s\left[ r(v) + \sigma(b, v-1)\right]
    \right\}.
  \end{align*}
\end{proof}

This reformulation will be useful for determining the monotonicity
properties of $\mu$. To do so, we will make use of the following
definition and theorem (a version of Topkis's Theorem
\cite{Topkis_MinSubmodular_1978}).
\begin{lemma}
  \label{lemma:topkis}
  Let $\Dcal_1 \subseteq \Rbb$ and $\Dcal_2 \subseteq \Rbb$ be
  non-empty and suppose $f: \Dcal_1 \times \Dcal_2 \goesto \Rbb$
  satisfies the following inequality for all $d_1^- \leq d_1^+$ and
  $d_2^- \leq d_2^+$:
  \begin{align*}
    f(d_1^+, d_2^+) &+ f(d_1^-, d_2^-) \leq f(d_1^+, d_2^-) + f(d_1^-, d_2^+)
  \end{align*}
  Then $f$ is \emph{submodular}. If $f$ is submodular and we define
  $g:\Dcal_2 \goesto \Dcal_1$ as
  \begin{equation*}
    g(d_2) = \min\argmin_{d_1 \in \Dcal_1} f(d_1, d_2)
  \end{equation*}
  then $g(\cdot)$ is non-decreasing.
\end{lemma}

\begin{proposition}
  \label{prop:g}
  There exists a non-decreasing function $g: \Rbb \goesto \Scal$ such
  that $\mu(b, v) = g(r(v) + \sigma(b, v-1))$.
\end{proposition}
\begin{proof}
  Let $f:\Scal \times \Rbb \goesto \Rbb$ be defined by $f(s, x) = c(s)
  - sx$. Take $s^+ \geq s^-$ and $x^+ \geq x^-$. $f$ is submodular if
  \begin{align*}
    f(s^+, x^+) + f(s^-, x^-) \leq  f(s^+, x^-) + f(s^-, x^+).
  \end{align*}
  Let $LHS$ and $RHS$ denote the left and right sides of the previous inequality.
  \begin{align*}
    LHS - RHS
    &= (c(s^+) - s^+ x^+ + c(s^-) - s^- x^-)\\
    &- (c(s^+) - s^+ x^- + c(s^-) - s^- x^+)\nonumber\\
    &= s^+ x^- + s^- x^+ - s^+ x^+ - s^- x^-\\
    &= (s^+ - s^-)(x^- - x^+)
  \end{align*}
  $(s^+ - s^-) \geq 0$ and $(x^- - x^+) \leq 0$ so $LHS \leq RHS$ and
  $f$ is submodular. Let $g$ be defined as
  \begin{equation*}
    g(x) = \min\argmin_{s \in \Scal} f(s, x)
  \end{equation*}
  By Lemma~\ref{lemma:topkis}, $g(\cdot)$ is non-decreasing and by
  Proposition \ref{prop:deltas}, $\mu(b, v) = g(r(v) + \sigma(b,
  v-1))$.
\end{proof}

The previous proposition shows that we can determine the monotonicity
properties of $\mu$ by understanding the monotonicity properties of
$r(v) + \sigma(b, v - 1)$. Since $r(v)$ does not depend on $b$, we can
study $b \mapsto \sigma(b, v)$ in order to understand $b \mapsto
\mu(b, v)$.

\begin{proposition}
  \label{prop:sigma}
  For each $(b, v) \in \Bcal \times \Vcal$, $T_{b, v}(\cdot)$ is
  non-decreasing and $\sigma(b, v) = T_{b, v}(\sigma(b, v-1))$.
\end{proposition}
\begin{proof}
  Take $x^+ \geq x^-$. For any $s \in \Scal$, $(1 - s) \geq 0$ so $(1
  - s)x^+ \geq (1-s)x^-$. Adding the same quantity to each side preserves the inequality so
  \begin{align*}
    h(b) + c(s) - r(v) &+ (1 - s)x^+ \geq\nonumber\\
    &h(b) + c(s) - r(v) + (1 - s)x^-
  \end{align*}
  Minimizing over $s \in \Scal$ and applying the monotonicity of
  minimization gives us that $T_{b,v}(x^+) \geq T_{b, v}(x^-)$.

  The second part of the proposition follows from the following algebraic manipulation:
  \begin{align*}
    \sigma&(b, v) \nonumber\\
    &= \sum_{i=0}^v \delta(b, i) = \delta(b, v) + \sigma(b, v-1)\\
    &= h(b)\\
    &+ \min_{s \in \Scal}\left\{c(s) - s\left[r(v) + \sum_{i=0}^{v-1}\delta(b, i)\right]\right\}
    + \sigma(b, v - 1)\nonumber\\
    &= h(b) \\
    &+ \min_{s \in \Scal}\left\{c(s) - s\left[r(v) + \sigma(b, v-1)\right]\right\}
    + \sigma(b, v - 1)\nonumber\\
    &= T_{b, v}(\sigma(b, v-1))
  \end{align*}
\end{proof}

\begin{theorem}
  For each $v \in \Vcal$, $b \mapsto \mu(b, v)$ is non-decreasing.
\end{theorem}
\begin{proof}
  We prove that $b \mapsto \sigma(b, v)$ is non-decreasing via
  induction. Because $\mu(b,v) = g(r(v) + \sigma(b, v))$ for some
  non-decreasing $g$, the result regarding $b \mapsto \mu(b, v)$
  follows immediately.

  For $v = 1$, 
  $$\sigma(b, v) = \delta(b, 1) = h(b) + \min_{s \in
    \Scal}\left\{c(s) - sr(1)\right\}.$$ By assumption, $h(\cdot)$ is
  non-decreasing so $b \mapsto \sigma(b, 1)$ is non-decreasing. Now
  assume that $b \mapsto \sigma(b, v)$ is non-decreasing for some $v
  \in \Vcal\setminus\set{V}$. Because $h(\cdot)$ is non-decreasing,
  $T_{b', v}(x) \geq T_{b, v}(x)$ whenever $b' \geq b$. In addition,
  $T_{b, v+1}(\cdot)$ is order-preserving (i.e. non-decreasing) and
  $\sigma(b, v+1) = T_{b, v+1}(\sigma(b, v))$. Therefore, $b \mapsto
  \sigma(b, v+1)$ is also non-decreasing. By induction, $b \mapsto
  \sigma(b, v)$ is non-decreasing for all $v \in \Vcal$.
\end{proof}

As demonstrated in Sec. \ref{sec:numerical}, the behavior of $v
\mapsto \mu(b, v)$ is slightly more nuanced. The following theorem
gives a set of algebraic conditions for determining the monotonicity
properties of $v \mapsto \mu(b, v)$. These conditions are useful and
interesting because they can be verified without computing
$\mu$. Furthermore, the proposition relates the rate of decay to the
$\delta$ terms. This matches our intuition that the rate of decay
should play a role in how the controller adapts to the decay itself.

\begin{theorem}
  Fix any $b \in \Bcal$.  If $\delta(b, v) \geq - [r(v+1) - r(v)]$ for
  all $v \in \Vcal\setminus\set{V}$, then $v \mapsto \mu(b, v)$ is
  non-decreasing.  If $\delta(b, v) \leq - [r(v+1) - r(v)]$ for all $v
  \in \Vcal\setminus\set{V}$, then $v \mapsto \mu(b, v)$ is
  non-increasing.
\end{theorem}
\begin{proof}
  Fix any $v \in \Vcal\setminus\set{V}$. By Proposition~\ref{prop:g}, 
  $\mu(b, v) = g(r(v) + \sigma(b, v-1))$ for some non-decreasing
  $g(\cdot)$. Therefore, $\mu(b, v+1) \geq \mu(b, v)$ if and only if
  $r(v + 1) + \sigma(b, v) \geq r(v) + \sigma(b, v - 1)$.
  \begin{align*}
    [r(v + 1) &+ \sigma(b, v)] - [r(v) + \sigma(b, v - 1)]\nonumber\\
    &= r(v + 1) - r(v) + [\sigma(b, v)  - \sigma(b, v-1)]\\
    &= r(v + 1) - r(v) + \delta(b, v)
  \end{align*}
  So if $\delta(b, v) \geq - [r(v + 1) - r(v)]$, then $\mu(b, v + 1)
  \geq \mu(b, v)$. If this holds for every $v \in
  \Vcal\setminus\set{V}$, then $v \mapsto \mu(b, v)$ is
  non-decreasing.

  The case for when $\delta(b, v) \leq - [r(v + 1) - r(v)]$ is
  analogous.
\end{proof}

When $r(v)$ is constant, we have an even simpler condition for testing
the monotonicity of $v \mapsto \mu(b, v)$. Taking $r(v)$ as a constant
can be used to model service time constraints; this was the case in
the wireless streaming model presented in \cite{Master_HOL_2014}. In
this case, $v \mapsto \mu(b, v)$ is always either non-decreasing or
non-increasing. A single algebraic condition can be verified to
determine which is the case.

\begin{theorem}
  Suppose $r(v) = \rbar > 0$ for all $v \in \Vcal$.  If
  $$h(b) + \min_{s \in \Scal} \{c(s) - s\rbar\} \geq 0$$
  then $v \mapsto \mu(b, v)$ is non-decreasing.  If 
  $$h(b) + \min_{s \in \Scal} \{c(s) - s\rbar\} \leq 0$$
  then $v \mapsto \mu(b, v)$ is non-increasing.  
\end{theorem}
\begin{proof}
  Define $T_{b, \rbar}:\Rbb \goesto \Rbb$ as follows:
  \begin{equation*}
    T_{b, \rbar}(x) 
    = x + h(b) + \min_{s \in \Scal} \{c(s) - s[\rbar +x]\}
  \end{equation*}
  Note that because $r(v) = \rbar$, $T_{b, \rbar}(x) = T_{b, v}(x)$
  for all $x \in \Rbb$. We are interested in the sign of $T_{b,
    \rbar}(0)$.

  Assume that $T_{b, \rbar}(0) \geq 0$. We show that $v \mapsto
  \sigma(b, v)$ is non-decreasing by applying the principle of
  mathematical induction. Since $\mu(b, v) = g(\rbar + \sigma(b,
  v-1))$ for some non-decreasing $g(\cdot)$, the result follows. The
  case of $T_{b, \rbar}(0) \leq 0$ is analogous.

  By Proposition~\ref{prop:sigma}, $\sigma(b, 1) = T_{b, \rbar}(0)$ and
  $\sigma(b, v + 1) = T_{b, \rbar}(\sigma(b, v))$ for all $v \in
  \Vcal\setminus\set{V}$. Applying $T_{b, \rbar}$ to $\sigma(b, 1) =
  T_{b, \rbar}(0) \geq 0$ and using the monotonicity of $T_{b,
    \rbar}(\cdot)$ gives us that
  \begin{equation*}
    \sigma(b, 2) = T_{b, \rbar}(\sigma(b, 1)) \geq T_{b, \rbar}(0) = \sigma(b, 1).
  \end{equation*}
  Now assume that $\sigma(b, v) \geq \sigma(b, v - 1)$ for some $v \in
  \Vcal \setminus\set{1}$. Then applying $T_{b, \rbar}$ to $\sigma(b,
  v) = T_{b, \rbar}(\sigma(b, v-1))$ and using the monotonicity of
  $T_{b, \rbar}(\cdot)$ gives us that
  \begin{equation*}
    \sigma(b, v + 1) = T_{b, \rbar}(\sigma(b, v)) 
    \geq T_{b, \rbar}(\sigma(b, v - 1)) = \sigma(b, v).
  \end{equation*}
  So by induction, if $T_{b, \rbar}(0) \geq 0$ then $\sigma(b, v + 1)
  \geq \sigma(b, v)$ for all $v \in \Vcal\setminus\set{V}$ and hence,
  $v \mapsto \sigma(b, v)$ is non-decreasing.
\end{proof}

\section{Future Work\label{sec:future}}
The results in this paper suggest a number of future modeling
extensions. For instance, we could consider jobs which have different
reward functions. This would make $r(v)$ into $r(v, b)$. In addition,
jobs could have different initial values so that instead of $V$ we
have $V(b)$. This could potentially lead to notational complications
because for $b < b'$ we might have that $\mu(b, v)$ is defined but
$\mu(b', v)$ is not. Having the initial value vary with the job would
create ``holes'' in the state space which could make it cumbersome to
discuss how the optimal policy varies with the number of remaining
jobs. On the other hand, allowing for these modeling extensions would
give more general results.

A more significant modeling extension would be including job
arrivals. The proofs in this paper take advantage of the fact that the
number of jobs in the system decreases over time. While it is
reasonable to conjecture that there are similar monotonicity
properties when job arrivals are included, the proofs in this paper
would need substantial modification to account for these properties.

\section{Conclusion\label{sec:conclusion}}
In this paper we have modeled a system in which jobs are completed by
a single server while a controller dynamically adjusts the service
rate. The reward for each job completion decays during service. Costs are
incurred for holding jobs and for exerting service effort. This can be
used as an abstract model for applications in healthcare, information
technology, as well as perishable inventory control.

We show that when the holding cost is non-decreasing, the optimal
policy will be non-decreasing in the number of remaining jobs. We also
give algebraic conditions for determining and verifying the
monotonicity of the optimal policy as a function of the residual
value. When the reward for job completion is given by a step function,
these algebraic conditions collapse into a single inequality that can
be used to determine the monotonicity of the optimal policy.

\bibliographystyle{ieeetr}
\bibliography{NMaster_NBambos_ServiceRate_ValueDecay}

\end{document}